\documentclass[11pt]{article}
\date{}

\usepackage[title]{appendix}
\usepackage{xcolor}
\usepackage[margin=1in]{geometry}                
\usepackage{graphicx}
\usepackage{subcaption}
\usepackage{pdflscape}
\usepackage{amssymb}
\usepackage[normalem]{ulem}
\usepackage{hyperref}
\usepackage{enumitem}
\usepackage{mathrsfs}
\usepackage{epstopdf}
\usepackage{rotating}
\usepackage{longtable} 
\usepackage{adjustbox}
\usepackage{float}

\usepackage{color}
\usepackage{bbm, dsfont}
\usepackage{pst-node}
\usepackage{tikz-cd}
\usepackage{amsfonts} 
\usepackage{amsthm}
\usepackage{amsmath}
\usepackage{titlesec}
\usepackage{amssymb}
\usepackage{enumitem}
\usepackage{float}
\usepackage [english]{babel}
\usepackage [autostyle, english = american]{csquotes}
\usepackage{algorithm}
\usepackage[noend]{algpseudocode} 
\makeatletter
\def\BState{\State\hskip-\ALG@thistlm}
\makeatother
\usepackage{hyperref}
\usepackage[normalem]{ulem}

\newlist{casess}{enumerate}{1}
\setlist[casess]{label=     \textbf{Case} \arabic*:}
\usepackage{mathtools}

\makeatletter
\newcommand*{\rom}[1]{\expandafter\@slowromancap\romannumeral #1@}
\makeatother

\usepackage{etoolbox}

\makeatletter
\patchcmd{\ttlh@hang}{\parindent\z@}{\parindent\z@\leavevmode}{}{}
\patchcmd{\ttlh@hang}{\noindent}{}{}{}
\makeatother

\usepackage{listings}
\usepackage{color} 
\definecolor{mygreen}{RGB}{28,172,0} 
\definecolor{mylilas}{RGB}{170,55,241}

\newlist{Assumptions}{enumerate}{1}
\setlist[Assumptions]{label=     \textbf{Assumption} \arabic*:}

\makeatletter

\newsavebox{\@brx}
\newcommand{\llangle}[1][]{\savebox{\@brx}{\(\m@th{#1\langle}\)}%
  \mathopen{\copy\@brx\kern-0.5\wd\@brx\usebox{\@brx}}}
\newcommand{\rrangle}[1][]{\savebox{\@brx}{\(\m@th{#1\rangle}\)}%
  \mathclose{\copy\@brx\kern-0.5\wd\@brx\usebox{\@brx}}}
\makeatother

\usepackage{lipsum} 
\usepackage{titlesec}
\titleformat{\subsection}[runin]
       {\normalfont\bfseries}
       {\thesubsection}
       {0.5em}
       {}
       [.]

\newcommand{\A}{\mathfrak{A}}

\newcommand{\B}{\mathfrak{B}}

\usepackage{tikz-cd}
\usepackage{graphicx}

\def\e{{\sf e}}

\def\d{{\rm d}}

\def\({\left(}
\def\[{\left[}
\def\){\right)}
\def\]{\right]}

\def\G{{\sf G}}
\def\K{{\sf K}}
\def\H{{\sf H}}

\def\<{\langle}
\def\>{\rangle}
\providecommand{\norm}[1]{\lVert#1\rVert}

 \newtheorem{thm}{Theorem}[section]
 \newtheorem{cor}[thm]{Corollary}
 
 \newtheorem{lem}[thm]{Lemma}
 \newtheorem{prop}[thm]{Proposition}
 \theoremstyle{definition}
 \newtheorem{defn}[thm]{Definition}
 \theoremstyle{remark}
 \newtheorem{rem}[thm]{Remark}
 \newtheorem{ex}[thm]{Example}
 \numberwithin{equation}{section}

\numberwithin{equation}{section}

\usepackage[backend=biber, maxnames=10]{biblatex}
\begin{document}


\title{Twisted convolution algebras with coefficients in a differential subalgebra}

\author{Felipe I. Flores}

\author{Felipe I. Flores
}

\maketitle

{\centering Dedicated to the memory of Marius Lauren\c{t}iu M\u{a}ntoiu (1961-2023).\par}


\begin{abstract}
Let $({\sf G},\alpha, \omega,\mathfrak B)$ be a measurable twisted action of the locally compact group ${\sf G}$ on a Banach $^*$-algebra $\mathfrak B$ and $\mathfrak A$ a differential Banach $^*$-subalgebra of $\mathfrak B$, which is stable under said action. We observe that $L^1_{\alpha,\omega}({\sf G},\mathfrak A)$ is a differential subalgebra of $L^1_{\alpha,\omega}({\sf G},\mathfrak B)$. We use this fact to provide new examples of groups with symmetric Banach $^*$-algebras. In particular, we prove that discrete rigidly symmetric extensions of compact groups are symmetric or that semidirect products ${\sf K}\rtimes{\sf H}$, with ${\sf H}$ symmetric and ${\sf K}$ compact, are symmetric.
\end{abstract}


\section{Introduction}\label{introduction}

The object of study for this paper is that of symmetric Banach $^*$-algebras (also called hermitian Banach $^*$-algebras by other authors), in particular for algebras associated to groups and their actions. Here are the precise definitions that will interest us.

\begin{defn} \label{symmetric}
A Banach $^*$-algebra $\mathfrak B$ is called {\it symmetric} if the spectrum of $b^*b$ is non-negative for every $b\in\mathfrak B$.
\end{defn}

In fact, the celebrated Shilari-Ford theorem says that a Banach $^*$-algebra $\mathfrak B$ is symmetric if and only if the spectrum of any self-adjoint element is real \cite[Theorem 11.4.1]{Pa01}. It will also be convenient to use the following standard terminology.

\begin{defn} 
A locally compact group $\G$ is called {\it symmetric} if the convolution algebra $L^1(\G)$ is symmetric.
\end{defn}

The study of symmetric Banach $^*$-algebras has a long history, most of which can be found in \cite[Chapters 10,11,12]{Pa01}. Besides being important in the general theory of Banach $^*$-algebras and spectral theory, the property of symmetry has implications on the K-theory and representation theory of the algebras that have it, making it a standing assumption in many works. It also links very strongly the algebra to its $C^*$-envelope, providing connections with the more robust study of $C^*$-algebra theory. As examples of this, we mention that it helps with the study of (algebraically) irreducible representations \cite[Lemma 3.13]{FJM}, makes the holomorphic functional calculus of the algebra the same as the one present in the universal $C^*$-completion \cite[Corollary 4]{GL06} and therefore the K-theories are equal \cite[Proposition 2.1]{Ro84}. Other applications happen in the study of time-frequency analysis \cite{GL06,GL04} and the study of automatic continuity.

The problem that concerns us here is that of determining which groups are symmetric. A characterization of the groups with this property is still missing, despite many advances. Some big classes of groups are known to have symmetric algebras: groups where all conjugacy classes are relatively compact \cite[Theorem 16]{Ku79}, compact extensions of nilpotent groups, connected groups of polynomial growth \cite{Lu79} and compactly generated groups of polynomial growth \cite[Corollary 1]{Lo01}. This last result could only be achieved by deeply describing the internal structure of compactly generated groups with polynomial growth, in a way that resembles Gromov's theorem \cite[Theorem 1]{Lo01}.

As for permanence properties, we know that open subgroups and quotients of symmetric groups are symmetric. The class of symmetric groups also includes semidirect products $\K\ltimes\H$, where $\K$ is compact and $\H$ is symmetric and central extensions where the central subgroup is open and the quotient is `rigidly symmetric' \cite{LP79}. Our work will provide a permanence property which is somewhat analogous to the last two, although our methods are quite different. It should be noted that the class of symmetric groups is very sensitive, so general (or even very nice) group extensions hardly stay inside the class. For example, Hulanicki constructed a non-symmetric group which is locally finite and $2$-solvable \cite{Hu80}.

Our ideas originate from the fact that many authors have studied the property of symmetry in the context of `generalized $L^1$-algebras' or even further generalizations (see for example \cite{FJM,Le77,LP79,Ma15,Py82}) and our main goal in this small article is to show new examples of symmetric groups by taking advantage of their investigations and the decomposition \begin{equation*}
    L^1(\G)\cong L^1_{\alpha,\omega}(\H, L^1(\K)),
\end{equation*} that occurs when $\G$ is an extension of $\K$ by $\H$, (using the kernel-by-quotient convention) developed in \cite{BS70}. We record our main result in the following theorem.

\begin{thm}
    Suppose that $\G$ is an extension of $\K$ by $\H$, with $\K$ compact and such that the quotient map $\pi:\G\to \H=\G/\K$ admits a Borel measurable section. Let $(\H,\alpha,\omega,L^1(\K))$ be the twisted action associated with this extension. Then $L^1_{\alpha,\omega}(\H, C^*(\K))$ is symmetric if and only if $\G$ is symmetric. If, in addition, we have $\omega\equiv1$, then $L^1_{\alpha}(\H, C^*(\K))$ and $L^1(\G)$ are symmetric as soon as $\H$ is symmetric.
\end{thm}

In order to prove this result, we relied heavily on the norm estimates that differential subalgebras provide -particularly $L^2(\K)$ inside $C^*(\K)$, for a compact group $\K$- and hence Section 2 is dedicated to a brief discussion on them, focusing on their properties related to spectral theory. We point out that their spectral invariance passes to convolution algebras (Proposition \ref{Coef}), something that does not occur in general (Example \ref{counter}). Section 3 then is entirely devoted to proving the general part (that is, the part involving the twist) of the main theorem and shows how to use the properties of differential subalgebras in a somewhat indirect way. We finish the article by analyzing the particular case of splitting extensions in Section \ref{suff}. This final section contains the last part of the theorem and its proof.

\section{Differential subalgebras}

Let $(\A,\norm{\cdot}_{\A})$, $(\B,\norm{\cdot}_{\B})$ be Banach $^*$-algebras, such that $\A$ is a Banach $^*$-subalgebra of $\B$, we assume that if $\A$ is unital, then $\B$ is also unital and they share the same unit. We say that $\A$ is a {\it differential subalgebra} of $\B$ if there is a constant $C>0$ such that 
\begin{equation}\label{diff}
    \norm{ab}_{\A}\leq C(\norm{a}_{\A}\norm{b}_{\B}+\norm{a}_{\B}\norm{b}_{\A}),\quad \textup{ for all } a,b\in\A.
\end{equation} The concept of differential subalgebras appeared first in the articles \cite{KS94,BC91}. One of the main uses of this property was found in \cite{GK14,GK13}, where it served to provide norm-estimates for the applications of the holomorphic functional calculus. Nature is abundant with examples of algebras with this property. The most interesting examples to us will be discussed now.

\begin{ex} \label{Schatten}
    Given a Hilbert space $\mathcal H$ and $p\in[1,\infty)$, the class of Schatten $p$-operators $S_p(\mathcal H)$ is a dense $^*$-ideal of the $C^*$-algebra of compact operators $\mathbb K(\mathcal H)$, moreover, one has $$\norm{ST}_{S_p(\mathcal H)}\leq\norm{S}_{\mathbb K(\mathcal H)}\norm{T}_{S_p(\mathcal H)}, \textup{ for all }S,T\in S_p(\mathcal H).$$ 
\end{ex}

\begin{ex} \label{hilbert}
    A proper $H^*$-algebra $\A$ is a Banach $^*$-algebra which is also a Hilbert space under the same norm and where $a^*$ is the unique element in $\A$ satisfying \begin{equation*}
        \langle ab,c\rangle=\langle b,a^*c\rangle\textup{ 
 and  }\langle ba,c\rangle=\langle b,ca^*\rangle\textup{, for all }b,c\in\A.
    \end{equation*} Then the left regular representation $\A\ni a\mapsto L_a\in \mathbb B(\A)$ is a contractive  $^*$-monomorphism. We then have $$\norm{ab}_{\A}\leq\norm{L_a}_{\mathbb B(\A)}\norm{b}_{\A}, \textup{ for all }a,b\in\A.$$ So that $\A$ is a differential subalgebra of $\B=\overline{L(A)}^{\norm{\cdot}_{\mathbb B(\A)}}\cong C^*(\A)$. In particular, when $\K$ is a compact group, equipped with the normalized Haar measure, $L^2(\K)$ becomes a proper $H^*$-algebra, with the operations inherited from $L^1(\K)$.
\end{ex}

\begin{rem}\label{remark}
    More examples include $C^1([a,b])$ on $C([a,b])$ \cite[page 470]{SS20}, some weighted subalgebras of $L^1(\G)$ \cite[Lemma 1]{Py82}, any full left Hilbert algebra on its $C^*$-completion \cite[Theorem 11.7.10]{Pa01}, or the algebras considered in Proposition \ref{Coef}.
\end{rem}

Let us now discuss one of the more important properties of differential subalgebras and the main reason why they interest us here. Naturally, ${\rm Spec}_{\A}(a)$ will denote the spectrum of an element $a\in\A$, while $\rho_\A(a)=\sup\{|\lambda|\mid \lambda\in {\rm Spec}_{\A}(a) \}$ is reserved to denote its spectral radius.

\begin{defn}
    Let $\A\subset\B$ be a continuous inclusion of Banach $^*$-algebras. We say that: 
    \begin{itemize}
        \item[(i)] $\A$ is {\it inverse-closed} in $\B$ if ${\rm Spec}_\A(a)={\rm Spec}_\B(a)$, for all $a\in\A$.
        \item[(ii)] $\A$ is {\it spectral radius preserving} in $\B$ if $\rho_\A(a)=\rho_\B(a)$, for all $a\in\A$.
    \end{itemize}
\end{defn}

It is clear that inverse-closedness can be rephrased (at least in the unital case) as `if $a\in\A$ is invertible inside of $\B$, then $a^{-1}\in \A$', which justifies the name. The next lemma is partially inspired by \cite[Lemma 3.2]{GK13}.

\begin{lem} \label{SRP}
    If $\A$ is a differential subalgebra of $\B$, then $\A$ is spectral radius preserving in $\B$. If $\A$ is also dense in $\B$, then it is spectrally invariant.
\end{lem}
\begin{proof}
   Observe that the inequality \ref{diff} implies that for all $a\in\A$, one has $$\norm{a^{2n}}_{\A}\leq 2C\norm{a^n}_{\A}\norm{a^n}_{\B}.$$ Since $\A$ and $\B$ are Banach algebras, Gelfand's formula for the spectral radius holds, so $$\rho_\A(a)=\lim_{n\to\infty} \norm{a^{2n}}_{\A}^{1/2n}\leq \lim_{n\to\infty} (2C)^{1/2n}\norm{a^n}_{\A}^{1/2n}\norm{a^n}_{\B}^{1/2n}= \rho_\A(a)^{1/2}\rho_\B(a)^{1/2},$$ hence $\rho_\A(a)=\rho_\B(a)$. For the second part, we follow the argument given in \cite[Proposition 2]{Ba00}: Suppose $a\in\A$ has an inverse $a^{-1}\in \B$. By density, there is a sequence $\{a_n\}_{n\in\mathbb N}\subset \A$ converging to $a^{-1}$ in $\B$. So $\lim_{n}\norm{1-aa_n}_\B=0$. In particular, for big enough $n$, one has 
   $$
   \rho_\A(1-aa_n)=\rho_\B(1-aa_n)\leq \norm{1-aa_n}_\B<1.
   $$ 
   Therefore, $aa_n$ is invertible in $\A$, and letting $c_n=(aa_n)^{-1}\in\A$, we see that $aa_nc_n=1$ and $a^{-1}=a_nc_n\in\A$.
\end{proof}

Let us now turn to the class of differential subalgebras we plan to study. First it becomes necessary to introduce twisted actions on Banach $^*$-algebras. If ${\B}$ is a Banach $^*$-algebra with isometric involution, we denote by ${\rm Aut}_1({\B})$ the group of its $^*$-automorphisms $L$ satisfying $\norm{L}=\norm{L^{-1}}=1$, by $\mathcal M({\B})$ its multiplier algebra and by $\mathcal{UM}({\B})$, the group of norm $1$ unitary multipliers. The following definition is taken from \cite{BS70}.

\begin{defn}\label{dynsis}
    A {\it (measurable) twisted action} is a $4$-tuple $(\G,\alpha,\omega,\A)$, where $\G$ is a locally compact group, $\A$ a Banach $^*$-algebra and $\alpha:\G\to{\rm Aut}_1({\A})$, $\omega:\G\times\G\to \mathcal{UM}({\A})$ are maps such that $\omega$ and $x\mapsto\alpha_x(a)$ are strongly Borel and satisfy \begin{itemize}
        \item[(i)] $\alpha_x(\omega(y,z))\omega(x,yz)=\omega(x,y)\omega(xy,z)$,
        \item[(ii)] $\alpha_x\big(\alpha_y(a)\big)\omega(x,y)=\omega(x,y)\alpha_{xy}(a)$,
        \item[(iii)] $\omega(x,\e)=\omega(\e,y)=1, \alpha_\e={\rm id}_{{\A}}$,
    \end{itemize} for all $x,y,z\in\G$ and $a\in\A$. $\e$ denotes the identity in $\G$. 
    When $\A$ is a $C^*$-algebra, $(\G,\alpha,\omega,\A)$ is also called a {\it twisted $C^*$-dynamical system} or just a {\it $C^*$-dynamical system} if the twist $\omega$ is trivial.
\end{defn} 

Given such a tuple, one usually forms the twisted convolution algebra $L^1_{\alpha,\omega}(\G,\A)$, consisting of all Bochner integrable functions $\Phi:\G\to\A$, endowed with the twisted convolution product 
\begin{equation}\label{convolution}
    \Phi*\Psi(x)=\int_\G \Phi(y)\alpha_y[\Psi(y^{-1}x)]\omega(y,y^{-1}x)\d y
\end{equation} and the involution \begin{equation}\label{involution}
    \Phi^*(x)=\Delta(x^{-1})\omega(x,x^{-1})^*\alpha_x[\Phi(x^{-1})^*].
\end{equation} Making $L^1_{\alpha,\omega}(\G,\A)$ a Banach $^*$-algebra under the norm $\norm{\Phi}_{L^1_{\alpha,\omega}(\G,\A)}=\int_\G\norm{\Phi(x)}_{\A}\d x$. When the twisted action is trivial (namely $\alpha\equiv {\rm id}_\A$ and $\omega\equiv1$), one recovers the projective tensor product: $L^1(\G,\A):=L^1_{{\rm id}_\A,1}(\G,\A)\cong L^1(\G)\Hat{\otimes}\A$.

\begin{defn}
    We say that a Banach $^*$-subalgebra $\A\subset \B$ is stable under the twisted action $(\G,\alpha,\omega,\B)$ if $\alpha_x(\A)\subset \A$ and $\omega(x,y)\big(\A\big)\subset \A$ for all $x,y\in \G$ and $(\G,\alpha,\omega,\A)$ is a twisted action.
\end{defn}

\begin{prop} \label{Coef}
    If $\,\A$ is a dense differential subalgebra of $\B$ and stable under the twisted action $(\G,\alpha,\omega,\B)$, then $L^1_{\alpha,\omega}(\G,\A)$ is a dense differential subalgebra of $L^1_{\alpha,\omega}(\G,\B)$. In particular, $L^1_{\alpha,\omega}(\G,\A)$ is inverse-closed in $L^1_{\alpha,\omega}(\G,\B)$.
\end{prop}
\begin{proof}
    It is clear that $L^1_{\alpha,\omega}(\G,\A)\hookrightarrow L^1_{\alpha,\omega}(\G,\B)$ in a natural way. This inclusion is dense, as $\B$-valued step functions can be approximated by $\A$-valued step functions. Finally, we have
    \begin{align*}
        \norm{\Phi*\Psi}_{L^1_{\alpha,\omega}(\G,\A)}&\leq \int_\G\int_\G \norm{\Phi(y)\alpha_y\big[\Psi(y^{-1}x)\big]}_{\A}\d y \d x \\
        &\leq C\int_\G\int_\G \norm{\Phi(y)}_{\A}\norm{\Psi(y^{-1}x)}_{\B}+\norm{\Phi(y)}_{\B}\norm{\Psi(y^{-1}x)}_{\A}\d y \d x \\
        &\leq C( \norm{\Phi}_{L^1_{\alpha,\omega}(\G,\A)}\norm{\Psi}_{L^1_{\alpha,\omega}(\G,\B)}+\norm{\Phi}_{L^1_{\alpha,\omega}(\G,\B)}\norm{\Psi}_{L^1_{\alpha,\omega}(\G,\A)}).
    \end{align*} Applying Lemma \ref{SRP} concludes the proof.\end{proof}

\begin{rem} \label{counter}
    Proposition \ref{Coef} does not hold under the weaker assumption that $\A$ is a (dense) inverse-closed $^*$-subalgebra of $\B$. For instance, \cite[Theorem 4]{LP79} guarantees the symmetry of $\A=L^1_{\rm lt}(\G,{C}_0(\G))$ but \cite[Theorem 5]{LP79} provides examples of non-symmetric algebras of the form $L^1_{\alpha}(\H,\A)$. This shows that, in general, $L^1_{\rm lt}(\G,{C}_0(\G))$ is an inverse-closed subalgebra but not a differential subalgebra of $\mathbb K(L^2(\G))$.
\end{rem}

In contrast with the previous remark, let us apply our results to some honest differential subalgebras of $\mathbb K(\mathcal H)$.

\begin{cor} \label{schat}
    Let $\G$ be a symmetric locally compact group. Then $L^1(\G,S_p(\mathcal H))\cong L^1(\G)\Hat{\otimes}S_p(\mathcal H)$ is a symmetric Banach $^*$-algebra.
\end{cor}
\begin{proof}
    $S_p(\mathcal H)$ is a dense differential subalgebra of $\mathbb K(\mathcal H)$, so Proposition \ref{Coef} and Lemma \ref{SRP} apply and we get $${\rm Spec}_{L^1(\G,S_p(\mathcal H))}(\Phi)={\rm Spec}_{L^1(\G,\mathbb K(\mathcal H))}(\Phi), \textup{ for all }\Phi\in L^1(\G,S_p(\mathcal H)).$$ But $L^1(\G,\mathbb K(\mathcal H))$ is symmetric when $L^1(\G)$ is symmetric \cite[Theorem 1]{Ku79}.
\end{proof}

\begin{rem}
    If $\mathfrak D$ is a symmetric Banach $^*$-algebra, then the projective tensor product $\mathfrak{D}\,\Hat{\otimes}{M}_n(\mathbb C)$ gives a symmetric Banach $^*$-algebra \cite[Theorem 11.4.2]{Pa01}. We see Corollary \ref{schat} (together with the symmetry of $L^1(\G)\Hat{\otimes}\mathbb K(\mathcal H)$) as an infinite-dimensional version of this fact.
\end{rem}

\begin{rem}
    Its easy to replicate the technique used in Corollary \ref{schat} in the cases mentioned in Remark \ref{remark}, under different assumptions. We will leave that to any reader interested in such applications.
\end{rem}

\section{On the extension of a compact group}\label{flocinor}

Let $\G$ be a (Hausdorff) locally compact group, $\K$ a normal, compact subgroup and set $\H=\G/\K$. We consider these groups to be fixed and the goal of this section is to use the methods found in the previous section to establish conditions on $\H$ that guarantee the symmetry of $L^1(\G)$. Following \cite{BS70}, we will describe how to decompose \begin{equation*}\label{decomp}
    L^1(\G)\cong L^1_{\alpha,\omega}(\H, L^1(\K)),
\end{equation*} for an appropriate twisted action $(\H,\alpha,\omega, L^1(\K))$. Indeed, we will suppose the existence of $\eta$, a Borel measurable right inverse of $\pi:\G\to \H=\G/\K$, with $\eta(\K)=\e$ and use it to define $$\tau(x,y)=\eta(x)\eta(y)\eta(xy)^{-1}\in\K,$$ for $x,y\in\H$. Then for $f\in L^1(\K)$ the relevant formulas that define the twisted action are
\begin{align}\label{action1}
    \alpha_x(f)(t)&=f\big(\eta(x)^{-1}t\eta(x)\big), \\ \label{action2}
    [\omega(x,y)f](t)=f(\tau(x,y)^{-1}t)&\textup{   and   } [f\omega(x,y)](t)=f(t\tau(x,y)^{-1}),
\end{align} for $t\in \K$ and $x,y\in\H$. 

It is fairly elementary to show that this twisted action extends to $C^*(\K)$ and also restricts to $L^2(\K)$ and more generally to $L^p(\K)$. Let us now prove the main theorem of the section.

\begin{thm} \label{main}
    Let $(\H,\alpha,\omega, L^1(\K))$ defined as above. If $L^1_{\alpha,\omega}(\H, C^*(\K))$ is a symmetric Banach $^*$-algebra, then $L^1(\G)\cong L^1_{\alpha,\omega}(\H, L^1(\K))$ is also symmmetric.
\end{thm}

\begin{proof}
    We first consider the (smaller) algebra $L^1_{\alpha,\omega}(\H, L^2(\K))$. $\A=L^2(\K)$ is a proper $H^*$-algebra, and hence a dense differential subalgebra of $C^*(\K)$. But $\A$ is also a dense differential subalgebra of $L^1(\K)$ and it then follows from Proposition \ref{Coef} that $L^1_{\alpha,\omega}(\H, L^2(\K))$ is inverse-closed in both $L^1_{\alpha,\omega}(\H, C^*(\K))$ and $L^1_{\alpha,\omega}(\H, L^1(\K))$, so we have \begin{equation*}
        {\rm Spec}_{L^1_{\alpha,\omega}(\H, L^1(\K))}(\Phi)={\rm Spec}_{L^1_{\alpha,\omega}(\H, C^*(\K))}(\Phi)={\rm Spec}_{L^1_{\alpha,\omega}(\H, L^2(\K))}(\Phi), 
    \end{equation*} for all $\Phi\in L^1_{\alpha,\omega}(\H, L^2(\K))$. In particular, we get that $L^1_{\alpha,\omega}(\H, L^2(\K))$ is a symmetric dense $^*$-ideal of $L^1_{\alpha,\omega}(\H, L^1(\K))$ and hence the latter has to be a symmetric Banach $^*$-algebra, due to \cite[page 86]{Wi78}, or \cite[Satz 1]{Le77}.
\end{proof}

The usefulness of Theorem \ref{main} lies on the fact that the study of the symmetry of $L^1$-algebras is far more developed -and easier to handle- in the case where the algebra of coefficients is a $C^*$-algebra. Let us mention some cases of interest.

\begin{ex} \label{hyper}
    The case of continuous twisted actions -that is, when $\alpha$ and $\omega$ are continuous maps- has received the most attention, so we can cite relevant results (for instance, this is always the case if $\H$ is discrete). In this case, the condition of symmetry in Theorem \ref{main} is always fulfilled when $\H$ is hypersymmetric, see \cite{FJM}. In particular, it applies when $\H$ is nilpotent.
    
    Another important case covered here could be the one of semidirect products $\G=\K\rtimes\H$. However, in this example more can be said and we will say it during the next section. As a spoiler: in Theorem \ref{semid} it is proved that is enough to only require $\H$ to be symmetric. It is known that there are symmetric but not hypersymmetric groups.
\end{ex}

As a corollary to the previous theorem, we explore different types of decay on the coefficient algebra. For $p\in[1,\infty]$, the Banach space $L^p(\K)$, of $p$-integrable functions, is a Banach $^*$-algebra when equipped with the operations \eqref{convolution}, \eqref{involution} and stable under the twisted action determined by \eqref{action1} and \eqref{action2}.

\begin{cor}\label{coro}
    Let $p\in[1,\infty]$. The algebra $L^1_{\alpha,\omega}(\H, L^p(\K))$ is symmetric if and only if $L^1_{\alpha,\omega}(\H, C^*(K))$ is symmetric.
\end{cor}

\begin{proof}
    Note that if $p\geq q$, then $L^p(\K)$ is a differential subalgebra of $L^q(\K)$, so by Proposition \ref{Coef}, one has \begin{equation}\label{spekk}
        {\rm Spec}_{L^1_{\alpha,\omega}(\H, L^p(\K))}(\Phi)={\rm Spec}_{L^1_{\alpha,\omega}(\H,L^q(\K))}(\Phi), \textup{ for all }\Phi\in L^1_{\alpha,\omega}(\H, L^p(\K)).
    \end{equation}
    If $L^1_{\alpha,\omega}(\H, C^*(K))$ is symmetric, then $L^1_{\alpha,\omega}(\H, L^1(\K))$ is symmetric because of Theorem \ref{main}, hence $L^1_{\alpha,\omega}(\H, L^p(\K))$ is symmetric for all $p$.
    
    On the other hand, if $L^1_{\alpha,\omega}(\H, L^p(\K))$ is symmetric and $p\geq 2$, then by the reasoning above $L^1_{\alpha,\omega}(\H, L^p(\K))$ is a symmetric dense $^*$-ideal of $L^1_{\alpha,\omega}(\H, L^2(\K))$, which makes the latter symmetric. But the latter algebra is also a symmetric dense $^*$-ideal of $L^1_{\alpha,\omega}(\H, C^*(\K))$, hence $L^1_{\alpha,\omega}(\H, C^*(\K))$ is symmetric. If $p<2$, then it follows immediately from \eqref{spekk} that $L^1_{\alpha,\omega}(\H, L^2(\K))$ is symmetric and we can repeat the last argument.
\end{proof}

    \section{Splitting extensions}\label{suff}

We will finish the paper by exploring symmetry in the particular case of splitting group extensions. The idea is to show that the symmetry of $\H$ provides a sufficient condition for the application of Theorem \ref{main} and thus, as this condition is also necessary, we obtain a complete understanding of when the extension is symmetric. It should be noted that this is a significant improvement with respect to the results we have for the general case (cf. Example \ref{hyper}). Obviously, we keep the notations and assumptions from the previous section. In particular,  we assume the existence of $\eta$, a Borel measurable right inverse of $\pi:\G\to \H=\G/\K$, with $\eta(\K)=\e$.

\begin{thm}\label{semid}
    Suppose that the section $\eta:\H\to \G$ can be chosen to be a group homomorphism. Then $L^1(\G)$ is symmetric whenever $L^1(\H)$ is symmetric.
\end{thm}
\begin{proof}
    It is clear that in this case, $\tau\equiv\e$ and consequently $\omega\equiv 1$ \eqref{action2}. Let $\lambda:C^*(\K)\to \mathbb B(L^2(\K))$ be the embedding obtained from the left regular representation. Then we define the unitary representation $U:\H\to \mathbb U(L^2(\K))$ given by \begin{equation*}
        U_x(f)(t)=f(\eta(x)^{-1}t\eta(x)), \textup{ for all }f\in L^2(\K).
    \end{equation*} Now observe that, for $f,g\in L^2(\K)$, one has \begin{align*}
        \big(U_x\lambda(f)U_x^*g\big)(t)&= \big(\lambda(f)U_x^*g\big)(\eta(x)^{-1}t\eta(x)) \\
        &=\int_\K f(s)U_x^*g(s^{-1}\eta(x)^{-1}t\eta(x))\d s \\
        &=\int_\K f(s)g(\eta(x)s^{-1}\eta(x)^{-1}t)\d s \\
        &=\int_\K \alpha_x(f)(s)g(s^{-1}t)\d s \\
        &=\lambda(\alpha_x(f))g(t)
    \end{align*}  Hence $U_x\lambda(f)U_x^*=\lambda(\alpha_x(f))$ and so the pair $(U,\lambda)$ is a covariant representation of $(\H,\alpha, C^*(\K))$. We now define an isometric $^*$-monomorphism \begin{equation*}
    \varphi:L_\alpha^1(\H,C^*(\K))\to L^1(\H,\mathbb K(L^2(\K))) \quad\textup{by}\quad\varphi(\Phi)(x)=\lambda(\Phi(x))U_x.
\end{equation*} Note that $\lambda(\Phi(x))\in C^*(\K)\subset \mathbb K(L^2(\K))$ and that makes $\varphi(\Phi)$ a Bochner measurable function. It is clear that $\varphi$ preserves the $L^1$-norm and it is therefore well-defined. We now verify its algebraic properties as a $^*$-homomorphism. In fact, we see that \begin{align*} \varphi(\Phi)*\varphi(\Psi)(x)&=\int_{\,\H}\lambda(\Phi(y))U_y\lambda(\Phi(y^{-1}x))U_{y^{-1}x}\d y \\
    &=\int_{\,\H}\lambda(\Phi(y))\lambda(\alpha_x(\Phi(y^{-1}x)))U_x \d y \\
    &=\lambda\Big(\int_{\,\H} \Phi(y)\alpha_x(\Phi(y^{-1}x))\d y \Big)U_x \\
    &=\varphi(\Phi*\Psi)(x)
\end{align*} and that $$\varphi(\Phi^*)(x)= \Delta(x^{-1})U_x\lambda(\Phi(x^{-1}))^*=\Delta(x^{-1})\varphi(\Phi)(x^{-1})^*=\varphi(\Phi)^*(x).$$ Therefore we can identify $L_\alpha^1(\H,C^*(\K))$ with a closed $^*$-subalgebra of $L^1(\H,\mathbb K(L^2(\K)))$. The latter is symmetric by \cite[Theorem 1]{Ku79}, so the former has to be symmetric. Then Theorem \ref{main} implies that $L^1(\G)$ is symmetric.
\end{proof} 

\begin{cor}
    In the setting of the previous theorem, we have that $\G$ is symmetric if and only if $\H$ is symmetric. 
\end{cor}

\begin{ex} \label{wreath}
    Let $\K$ be a compact group and $\H$ be a discrete group acting on a set $\Omega$. We denote by $\K^\Omega$ the direct product of $\K$ with itself, indexed by $\Omega$. We then can form the group $\G=\K^{\Omega}\rtimes\H$, where $\H$ acts by shifting coordinates and Theorem \ref{semid} then implies that this group is symmetric exactly when $\H$ is symmetric. 
\end{ex}

\section*{Acknowledgements}

The author has been partially supported by the NSF grant DMS-2000105. He is also grateful to professor Ben Hayes for his very helpful comments.

\printbibliography

@article {KS94,
    AUTHOR = {Kissin, E. and Shul\textquotesingle man, V. S.},
     TITLE = {Differential properties of some dense subalgebras of
              {$C^\ast$}-algebras},
   JOURNAL = {Proc. Edinburgh Math. Soc. (2)},
  FJOURNAL = {Proceedings of the Edinburgh Mathematical Society. Series II},
    VOLUME = {37},
      YEAR = {1994},
    NUMBER = {3},
     PAGES = {399--422},
      %issn = {0013-0915,1464-3839},
   MRCLASS = {46L87 (46K10 46L57 47D25)},
  MRNUMBER = {1297311}
}

@article {BC91,
    AUTHOR = {Blackadar, B. and Cuntz, J.},
     TITLE = {Differential {B}anach algebra norms and smooth subalgebras of
              {$C^*$}-algebras},
   JOURNAL = {J. Operator Theory},
  FJOURNAL = {Journal of Operator Theory},
    VOLUME = {26},
      YEAR = {1991},
    NUMBER = {2},
     PAGES = {255--282},
      %issn = {0379-4024},
   MRCLASS = {46L87 (46K10 46L57 46M99)},
  MRNUMBER = {1225517},
MRREVIEWER = {C.\ J. K. Batty},
}

@article {GK13,
    AUTHOR = {Gr\"{o}chenig, Karlheinz and Klotz, Andreas},
     TITLE = {Norm-controlled inversion in smooth {B}anach algebras, {I}},
   JOURNAL = {J. Lond. Math. Soc. (2)},
  FJOURNAL = {Journal of the London Mathematical Society. Second Series},
    VOLUME = {88},
      YEAR = {2013},
    NUMBER = {1},
     PAGES = {49--64},
      %issn = {0024-6107,1469-7750},
   MRCLASS = {46H30 (41A65 47A60)},
  MRNUMBER = {3092257}
}

@article {GK14,
    AUTHOR = {Gr\"{o}chenig, Karlheinz and Klotz, Andreas},
     TITLE = {Norm-controlled inversion in smooth {B}anach algebras, {II}},
   JOURNAL = {Math. Nachr.},
  FJOURNAL = {Mathematische Nachrichten},
    VOLUME = {287},
      YEAR = {2014},
    NUMBER = {8-9},
     PAGES = {917--937},
      %issn = {0025-584X,1522-2616},
   MRCLASS = {46H30 (41A65 47A60 47B47)},
  MRNUMBER = {3219221},
}

@article {Ma15,
    AUTHOR = {M\u{a}ntoiu, Marius},
     TITLE = {Symmetry and inverse closedness for {B}anach *-algebras
              associated to discrete groups},
   JOURNAL = {Banach J. Math. Anal.},
  FJOURNAL = {Banach Journal of Mathematical Analysis},
    VOLUME = {9},
      YEAR = {2015},
    NUMBER = {2},
     PAGES = {289--310},
      %issn = {2662-2033,1735-8787},
   MRCLASS = {47L65 (22D15 43A20)},
  MRNUMBER = {3296119},
}

@article {Py82,
    AUTHOR = {Pytlik, T.},
     TITLE = {Symbolic calculus on weighted group algebras},
   JOURNAL = {Studia Math.},
  FJOURNAL = {Polska Akademia Nauk. Instytut Matematyczny. Studia
              Mathematica},
    VOLUME = {73},
      YEAR = {1982},
    NUMBER = {2},
     PAGES = {169--176},
      %issn = {0039-3223,1730-6337},
   MRCLASS = {43A20 (46H30)},
  MRNUMBER = {667971},
}

@article {BS70,
    AUTHOR = {Busby, Robert C. and Smith, Harvey A.},
     TITLE = {Representations of twisted group algebras},
   JOURNAL = {Trans. Amer. Math. Soc.},
  FJOURNAL = {Transactions of the American Mathematical Society},
    VOLUME = {149},
      YEAR = {1970},
     PAGES = {503--537},
      %issn = {0002-9947,1088-6850},
   MRCLASS = {46.80 (42.56)},
  MRNUMBER = {264418},
}

@article {LP79,
    AUTHOR = {Leptin, Horst and Poguntke, Detlev},
     TITLE = {Symmetry and nonsymmetry for locally compact groups},
   JOURNAL = {J. Functional Analysis},
  FJOURNAL = {Journal of Functional Analysis},
    VOLUME = {33},
      YEAR = {1979},
    NUMBER = {2},
     PAGES = {119--134},
      %issn = {0022-1236},
   MRCLASS = {43A20},
  MRNUMBER = {546502}
}

@article{FJM,
    AUTHOR = {Flores, Felipe and Jaur\'e, Diego and M\u antoiu, Marius},
     TITLE = {Symmetry for algebras associated to Fell bundles over groups and groupoids},
    JOURNAL = {J. Operator Theory},
  FJOURNAL = {Journal of Operator Theory},
VOLUME = {91},
      YEAR = {2024},
    NUMBER = {1},
     PAGES = {27--54},
}

@article {Ku79,
    AUTHOR = {Kugler, Werner},
     TITLE = {On the symmetry of generalized {$L\sp{1}$}-algebras},
   JOURNAL = {Math. Z.},
  FJOURNAL = {Mathematische Zeitschrift},
    VOLUME = {168},
      YEAR = {1979},
    NUMBER = {3},
     PAGES = {241--262},
      %issn = {0025-5874,1432-1823},
   MRCLASS = {46M05 (43A20 46H05 46L05)},
  MRNUMBER = {544593},
}

@article {Ba00,
    AUTHOR = {Barnes, Bruce A.},
     TITLE = {Symmetric {B}anach {$\ast$}-algebras: invariance of spectrum},
   JOURNAL = {Studia Math.},
  FJOURNAL = {Studia Mathematica},
    VOLUME = {141},
      YEAR = {2000},
    NUMBER = {3},
     PAGES = {251--261},
      %issn = {0039-3223,1730-6337},
   MRCLASS = {46K05 (46H35)},
  MRNUMBER = {1784672},
}

@book{Pa01,
    author = {Palmer, Theodore W.},
    title = {Banach algebras and the general theory of \,$^*$-algebras: Volume 2, $^*$-algebras},
    publisher = {Vol. 2. Cambridge university press},
    year = {2001}, 
}

@article {Hu80,
    AUTHOR = {Hulanicki, A.},
     TITLE = {Invariant subsets of nonsynthesis {L}eptin algebras and
              nonsymmetry},
   JOURNAL = {Colloq. Math.},
  FJOURNAL = {Colloquium Mathematicum},
    VOLUME = {43},
      YEAR = {1980},
    NUMBER = {1},
     PAGES = {127--136},
}

@incollection {SS20,
    AUTHOR = {Shin, Chang Eon and Sun, Qiyu},
     TITLE = {Differential subalgebras and norm-controlled inversion},
 BOOKTITLE = {Operator theory, operator algebras and their interactions with
              geometry and topology---{R}onald {G}. {D}ouglas memorial
              volume},
    SERIES = {Oper. Theory Adv. Appl.},
    VOLUME = {278},
     PAGES = {467--485},
 PUBLISHER = {Birkh\"{a}user/Springer, Cham},
      YEAR = {2020},
      %isbn = {978-3-030-43380-2; 978-3-030-43379-6},
   MRCLASS = {46H05 (46K05 47B90)},
  MRNUMBER = {4200266},
}

@incollection {Le77,
    AUTHOR = {Leptin, H.},
     TITLE = {Lokal kompakte {G}ruppen mit symmetrischen {A}lgebren},
 BOOKTITLE = {Symposia {M}athematica, {V}ol. {XXII} ({C}onvegno
              sull'{A}nalisi {A}rmonica e {S}pazi di {F}unzioni su {G}ruppi
              {L}ocalmente {C}ompatti, {INDAM}, {R}ome, 1976)},
     PAGES = {267--281},
 PUBLISHER = {Academic Press, London-New York},
      YEAR = {1977},
   MRCLASS = {22D15 (43A10)},
  MRNUMBER = {486296},
}

@article {Wi78,
    AUTHOR = {Wichmann, Josef},
     TITLE = {The symmetric radical of an algebra with involution},
   JOURNAL = {Arch. Math. (Basel)},
  FJOURNAL = {Archiv der Mathematik},
    VOLUME = {30},
      YEAR = {1978},
    NUMBER = {1},
     PAGES = {83--88},
      %ISSN = {0003-889X,1420-8938},
   MRCLASS = {46K05},
  MRNUMBER = {482233},
MRREVIEWER = {Allan\ M.\ Sinclair},
}

@article {Lo01,
    AUTHOR = {Losert, V.},
     TITLE = {On the structure of groups with polynomial growth. {II}},
   JOURNAL = {J. London Math. Soc. (2)},
  FJOURNAL = {Journal of the London Mathematical Society. Second Series},
    VOLUME = {63},
      YEAR = {2001},
    NUMBER = {3},
     PAGES = {640--654},
      %ISSN = {0024-6107,1469-7750},
   MRCLASS = {22D05 (20F24 22E25)},
  MRNUMBER = {1825980},
}

@article {Lu79,
    AUTHOR = {Ludwig, Jean},
     TITLE = {A class of symmetric and a class of {W}iener group algebras},
   JOURNAL = {J. Functional Analysis},
  FJOURNAL = {Journal of Functional Analysis},
    VOLUME = {31},
      YEAR = {1979},
    NUMBER = {2},
     PAGES = {187--194},
      %ISSN = {0022-1236},
   MRCLASS = {43A20 (22D15)},
  MRNUMBER = {525950},
}

@article {GL04,
    AUTHOR = {Gr\"ochenig, Karlheinz and Leinert, Michael},
     TITLE = {Wiener's lemma for twisted convolution and {G}abor frames},
   JOURNAL = {J. Amer. Math. Soc.},
  FJOURNAL = {Journal of the American Mathematical Society},
    VOLUME = {17},
      YEAR = {2004},
    NUMBER = {1},
     PAGES = {1--18},
}

@article {GL06,
    AUTHOR = {Gr\"ochenig, Karlheinz and Leinert, Michael},
     TITLE = {Symmetry and inverse-closedness of matrix algebras and
              functional calculus for infinite matrices},
   JOURNAL = {Trans. Amer. Math. Soc.},
  FJOURNAL = {Transactions of the American Mathematical Society},
    VOLUME = {358},
      YEAR = {2006},
    NUMBER = {6},
     PAGES = {2695--2711},
}

@incollection {Ro84,
    AUTHOR = {Rosenberg, J.},
     TITLE = {Group {$C\sp{\ast} $}-algebras and topological invariants},
 BOOKTITLE = {Operator algebras and group representations, {V}ol. {II}
              ({N}eptun, 1980)},
    SERIES = {Monogr. Stud. Math.},
    VOLUME = {18},
     PAGES = {95--115},
 PUBLISHER = {Pitman, Boston, MA},
      YEAR = {1984},
}

\bigskip
\bigskip
ADDRESS

\smallskip
Felipe I. Flores

Department of Mathematics, University of Virginia,

114 Kerchof Hall. 141 Cabell Dr,

Charlottesville, Virginia, United States

E-mail: hmy3tf@virginia.edu

\end{document}